\documentclass[12pt]{amsart}
\usepackage{amsfonts,amssymb,amscd,amsmath,enumerate,verbatim}
\usepackage[latin1]{inputenc}
\usepackage{amscd}
\usepackage{latexsym}
\usepackage{pstcol,pst-plot,pst-3d}
\usepackage{mathptmx}
\usepackage{multicol}
\usepackage{setspace}

\psset{unit=0.7cm,linewidth=0.8pt,arrowsize=2.5pt 4}
\newpsstyle{fatline}{linewidth=1.5pt}
\newpsstyle{fyp}{fillstyle=solid,fillcolor=verylight}
\definecolor{verylight}{gray}{0.97}
\definecolor{light}{gray}{0.9}
\definecolor{medium}{gray}{0.85}

\def\frk{\mathfrak}               % font for "Fraktur"

\def\Phi{{\frk N}}
\def\opn#1#2{\def#1{\operatorname{#2}}} % to make operators
\opn\chara{char} \opn\length{\ell} \opn\pd{pd} \opn\rk{rk}
\opn\projdim{proj\,dim} \opn\injdim{inj\,dim} \opn\rank{rank}
\opn\depth{depth} \opn\grade{grade} \opn\height{height} \opn\bheight{bigheight}
\opn\embdim{emb\,dim} \opn\codim{codim}

\opn\Tr{Tr} \opn\bigrank{big\,rank}
\opn\superheight{superheight}\opn\lcm{lcm}
\opn\trdeg{tr\,deg}%\emph{
\opn\reg{reg} \opn\lreg{lreg} \opn\ini{in} \opn\lpd{lpd}
\opn\size{size}\opn{\mult}{mult} \opn{\rev}{rev}
\opn\div{div} \opn\Div{Div} \opn\cl{cl} \opn\Cl{Cl}
\opn\Spec{Spec} \opn\Supp{Supp} \opn\supp{supp} \opn\Sing{Sing}
\opn\Ass{Ass} \opn\Min{Min}
\opn\Ann{Ann} \opn\Rad{Rad} \opn\Soc{Soc}
\opn\Syz{Syz} \opn\Im{Im} \opn\Ker{Ker} \opn\Coker{Coker}
\opn\Am{Am} \opn\Hom{Hom} \opn\Tor{Tor} \opn\Ext{Ext}
\opn\End{End} \opn\Aut{Aut} \opn\id{id} \opn\ini{in}

\opn\nat{nat}
\opn\pff{pf}%   \pf exists already
\opn\Pf{Pf} \opn\GL{GL} \opn\SL{SL} \opn\mod{mod} \opn\ord{ord}
\opn\Gin{Gin}
\opn\Hilb{Hilb}\opn\adeg{adeg}\opn\std{std}\opn\ip{infpt}
\opn\Pol{Pol}
\opn\sat{sat}
\opn\Var{Var}
\opn\Gen{Gen}
\opn\indmatch{indmatch}

\opn\aff{aff} \opn\con{conv} \opn\relint{relint} \opn\st{st}
\opn\lk{lk} \opn\cn{cn} \opn\core{core} \opn\vol{vol}
\opn\link{link} \opn\star{star}
\opn\gr{gr}

\def\pot#1#2{#1[\kern-0.28ex[#2]\kern-0.28ex]}

\opn\dirlim{\underrightarrow{\lim}}
\opn\inivlim{\underleftarrow{\lim}}
\def\Implies{\ifmmode\Longrightarrow \else
        \unskip${}\Longrightarrow{}$\ignorespaces\fi}
\def\implies{\ifmmode\Rightarrow \else
        \unskip${}\Rightarrow{}$\ignorespaces\fi}
\def\iff{\ifmmode\Longleftrightarrow \else
        \unskip${}\Longleftrightarrow{}$\ignorespaces\fi}

\let\:=\colon
\newtheorem{Theorem}{Theorem}[section]
\newtheorem{Lemma}[Theorem]{Lemma}
\newtheorem{Corollary}[Theorem]{Corollary}
\newtheorem{Proposition}[Theorem]{Proposition}

\newtheorem{Example}[Theorem]{Example}

\let\epsilon\varepsilon
\let\phi=\varphi
\let\kappa=\varkappa
\def\qed{\ifhmode\textqed\fi
      \ifmmode\ifinner\quad\qedsymbol\else\dispqed\fi\fi}
\def\textqed{\unskip\nobreak\penalty50
       \hskip2em\hbox{}\nobreak\hfil\qedsymbol
       \parfillskip=0pt \finalhyphendemerits=0}
\def\dispqed{\rlap{\qquad\qedsymbol}}

\opn\dis{dis}
\def\pnt{{\raise0.5mm\hbox{\large\bf.}}}

\opn\Lex{Lex}

\newcommand{\inD}[1][\relax]{\def\argone{#1}\def\temprelax{\relax}
  \ifx\argone\temprelax\right.\else\,\middle|#1\right.{}\fi}

\newif\ifbinary
\binarytrue
\begin{document}

\title{Gorenstein binomial edge ideals associated with scrolls}

\author{ Ahmet Dokuyucu, Ajdin Halilovic, Rida Irfan}

\thanks{The third author acknowledges the support from Higher Education Commission of Pakistan.}
%\subjclass{}

\address{Ahmet Dokuyucu,  Lumina-The University of South-East Europe,
Sos. Colentina nr. 64b, Bucharest,
Romania} \email{ahmet.dokuyucu@lumina.org}

\address{Ajdin Halilovic,  Lumina-The University of South-East Europe,
Sos. Colentina nr. 64b, Bucharest,
Romania} \email{ajdin.halilovic@lumina.org}

\address{Rida Irfan, Abdus Salam School of Mathematical Sciences, GC University,
68-B, New Muslim Town, Lahore 54600, Pakistan} \email{ridairfan\_88@yahoo.com}

\begin{abstract}

Let $I_G$ be the binomial edge ideal on the generic $2\times n$ - Hankel matrix associated with  a closed graph $G$ on the vertex set $[n]$. We characterize the graphs $G$ 
for which $I_G$ has maximal regularity and is Gorenstein.
\end{abstract}
\maketitle

\section*{Introduction}
Let $K$ be a field and $S=K[x_1,\dots,x_n,x_{n+1}]$ the polynomial ring in $n+1$ variables over the field $K$. 
Let $X=\begin{pmatrix}
x_{1}  & x_{2}   & \cdots     &  x_{n}      \\
x_{2}  & x_{3}   & \cdots     &  x_{n+1} 
\end{pmatrix}$
be the generic $2\times n$ - Hankel matrix and $G$ a closed graph on the vertex set $[n]$, that is, a graph satisfying the following condition: there exists a labeling of $G$ with the property that if $\{i,j\}$ and $\{i,k\}$ are edges of $G$ such that either $i<j<k$ or $i>j>k$, then $\{j,k\}$ is an edge of $G$.

Closed graphs were introduced in \cite{HHHKR} in order to characterize  binomial edge ideals which have a quadratic Gr\"{o}bner basis. However, it turns out that these graphs were already known in combinatorics as indifference graphs. Namely, by \cite[Theorem 1]{LO}, a graph $G$ is an indifference graph if and only if for every edge $\{i,k\}$ of $G$ and every $j$ with $i<j<k$, also $\{i,j\}$ and $\{j,k\}$ are edges of $G$. The latter property can be easily shown to be equivalent to the definition of a closed graph. On the other hand, the notion of indifference graph is equivalent to the notion of proper interval graph \cite{Ro}. Indifference graphs and, more general, interval graphs have been intensively studied from combinatorial and algorithmic point of view; see \cite{LO} and the references therein. 

In this paper, we will use the terminology closed graph.

 In \cite{CDE} there was considered the ideal $I_G\subset S$ which is generated by all the 2-minors $g_{ij}=
\begin{vmatrix}
x_{i}  & x_{j}      \\
x_{i+1}  & x_{j+1}  
\end{vmatrix}$
of $X$ which correspond to the edges $\{i,j\}$ of $G$. The ideal $I_G$ is a natural generalization of the ideal $I_C$ of the rational normal curve 
$C\subset\mathbb{P}^n$. Indeed, if $G=K_n$, then $I_G=I_C$. The minimal free resolution of $S/I_C$ is the Eagon-Northcott resolution.

In \cite{CDE} it was shown that, for any closed graph $G$, the ideal $I_G$ has a quadratic Gr\"obner basis with respect to the reverse lexicographic 
order on $S$ induced by $x_1>\cdots >x_{n+1}$ and that $I_G$ is Cohen-Macaulay of dimension $1+c$ where $c$ is the number of connected components of $G$. 

In the same paper it was shown that the regularity of $S/I_G$ is bounded above by the number of maximal cliques of the graph $G$. For a graph $G$, the collection of cliques of $G$ (i.e. the complete subgraphs of $G$) forms a simplicial complex $\Delta(G)$ which is called the clique complex of $G$. We recall from \cite{EHH} that $G$ is a closed graph if and only if there exists a labeling of $G$ such that all facets of $\Delta(G)$ are intervals.

Let $G$ be a closed graph on the vertex set $[n]$ with $\Delta(G)=\langle F_1,\dots,F_r\rangle$ where $ F_{i}=[a_{i},b_{i}]$ for $1\leq i\leq r$ and 
$1=a_1<a_2<\cdots <a_r<b_r=n$. Then, as it was shown in \cite{CDE}, we have reg$(S/I_G)\leq r$. For the closed graphs $G$ which satisfy the conditions $a_{i+1}=b_i$ for $1\leq i\leq r-1$, it was shown that reg$(S/I_G)$ is exactly $r$.

In Section 1 we characterize the closed graphs $G$ for which reg$(S/I_G)$ is equal to $r$. Note that, since $I_G$ and the initial ideal of $I_G$ with respect to the reverse lexicographic order, $\ini_{\rev}(I_G)$, are both Cohen-Macaulay, we have reg$(S/I_G) = \reg(S/\ini_{\rev}(I_G)) =$ deg$P(t)$, where $P(t)$ is the numerator polynomial of the Hilbert series $H_{S/I_G}(t)=H_{S/\ini_{\rev}(I_G)}(t)$. In Theorem \ref{maxreg} we show that $I_G$ has maximal regularity if and only if any three consecutive maximal cliques of $G$ have empty intersection.

From  combinatorial point of view, the proof of Theorem \ref{maxreg} is quite simple. Once we are given the intervals $I_1,\dots,I_r$, where $I_{j}=[a_{j}+1,b_{j}]$, for $1\leq j\leq r$, we may consider the simplicial complex $\Sigma$ of all subsets $\sigma\subset\{2,\dots,n\}$ which contain at most one element from each interval $I_j$. The faces of $\Sigma$ are in one-to-one correspondence with the monomials that form a vector space basis of the algebra $S/(\ini_{\rev}(I_G),x_1,x_{n+1})$. The proof of Theorem \ref{maxreg} actaully answers the following combinatorial question: when does there exist $\sigma\in\Sigma$ of cardinality $r$? The algebraic interpretation of this question is: when does $I_G$ have maximal regularity?

In Section 2 we state and prove the main theorem of this paper which characterizes the closed graphs $G$ for which $I_G$ is a Gorenstein ideal. To this aim, in Lemma~\ref{reg_is_r} we first show for connected closed graphs $G$ that, if $I_G$ is Gorenstein and it has maximal regularity, then the following numerical conditions must hold: 
\[
a_2=2, a_{i+2}=b_i+1 \text{~for~} 1\leq i\leq r-2, \text{~and~} b_{r-1}=n-1.
\]

In Theorem \ref{mainthm} we then characterize all connected closed graphs $G$ for which $I_G$ is Gorenstein. Finally, in Proposition \ref{congor} we show that $I_G$ is Gorenstein if and only if the associated ideal of each connected component of $G$ is Gorenstein. Thus, a complete characterization of closed graphs $G$ for which $I_G$ is a Gorenstein ideal is given. The proof uses in principal combinatorial techniques.

Coming back to the above combinatorial interpretation of the regularity of $I_G$, note that imposing the condition that $\Sigma$ has a single facet of cardinality $r$ does not solve  our problem, since the Gorensteiness of $I_G$ does not automatically imply maximal regularity. This fact makes the proof of Theorem \ref{mainthm} much more complicated.

\section{Scroll binomial edge ideals of maximal regularity}\label{Rida}

Let $G$ be a closed graph on the vertex set $[n]$ and with the clique complex $\Delta (G) =\langle F_{1},\ldots, F_{r}\rangle$ where $ F_{i}=[a_{i},b_{i}]$, for $1\leq i\leq r$, and $1=a_{1} < a_{2} < \cdots < a_{r} < b_{r}=n$.
In \cite{CDE} it was shown that $\reg S/I_{G}\leq r$. Moreover, in the same paper it was shown that if the cliques of $G$ satisfy the conditions $a_{i+1}=b_{i}$ for $1\leq i \leq r-1$, then $\reg S/I_{G}=r$. In this section we  give a full characterization of the graphs $G$ with  the property that $\reg S/I_{G}=r$.

Before giving this characterization, we prove a nice property of the graphs considered in \cite{CDE}. The classical binomial edge ideals share a similar property; see \cite[Proposition~3.2]{EHH}.

\begin{Proposition}\label{bettieq}
Let $G$ be a closed graph with the maximal cliques $F_{i}=[a_{i},b_{i}]$ for $1\leq i\leq r$. If $a_{i+1}=b_{i}$ for $1\leq i \leq r-1$, then, for all $i,j$, we have \[\beta_{ij}(S/I_{G})=\beta_{ij}(S/\ini_{\rev}( I_{G})).\]
\end{Proposition}

\begin{proof}
We essentially follow the proof of \cite[Proposition~3.2]{EHH}.

For a graded $S$-module $M$, let $B_{M}(s,t)= \sum_{i,j} \beta_{ij}s^i t^j$ be the Betti polynomial of $M$. We have
$\ini_{\rev}(I_{G})=(x_{2},\ldots,x_{a_{2}})^2 + (x_{a_{2}+1},\ldots,x_{a_{3}})^2 +\cdots + (x_{a_{r}+1},\ldots,x_{n})^2$. Let $M_{i}$ be the minimal monomial generating set of $(x_{a_{i}+1},\ldots,x_{a_{i+1}})^2$ for $1\leq i \leq r$, where $a_{1}=1$ and $a_{r+1}=n$. Then, for any $i\neq j$, we have $M_{i}\cap M_{j}=\emptyset$. It follows that
\begin{equation}\label{*}
 \Tor_{k}(S/(M_{i}),S/(M_{j}))=0 \text{ for all }i\neq j \text{ and }k> 0.
\end{equation}

This implies that $$ B_{S/\ini_{\rev}(I_{G})}(s,t)=\prod_{i=1} ^r B_{S/(M_i)}(s,t).$$
On the other hand, by \cite[Proposition 3.13]{BC}, relation (\ref{*}) implies that $\Tor_{k}(S/I_{F_{i}},S/I_{F_{j}})=0$ for all $i\neq j$ and $k>0$. Therefore, we get
\begin{equation}\label{**}
B_{S/I_{G}}(s,t)=\prod_{i=1} ^r B_{S/I_{F_i}}(s,t).
\end{equation}
In order to prove our statement, it is enough to show that, if $G$ consists of a single clique, then $\beta_{ij}(S/I_{G})=\beta_{ij}(S/\ini_{\rev}(I_{G}))$. But, if $G$ is a clique on the vertex set $[n]$, then $\ini_{\rev}(I_{G})=(x_{2},\ldots,x_{n})^2$ has a linear resolution. Then $I_{G}$ has a linear resolution as well. Consequently, the Hilbert series of $S/I_{G}$ and $S/\ini_{\rev}(I_{G})$ are determined by the corresponding Betti numbers. As $S/I_{G}$ and $S/\ini_{\rev}(I_{G})$ have the same Hilbert series, it follows that $ \beta_{ij}(S/I_{G})=\beta_{ij}(S/\ini_{\rev}(I_{G}))$ for all $i,j$.
\end{proof}

In what follows, we characterize the graphs $G$ whose associated ideal $I_{G}$ has a maximal regularity. First, we show that we may reduce to the connected case.

Let $G$ be a closed graph on the vertex set $[n]$ and with the connected components $G_{1},\ldots,G_{c}$. Let $r_{i}$ be the number of cliques of $G_{i}$ for $1\leq i\leq c$, and $r=r_{1}+\cdots+r_{c}$. By the proof of \cite[Theorem 2.7]{CDE}, it follows that 
\[
\reg(S/I_G)\leq \sum_{i=1}^c\reg(S/I_{G_i})\leq \sum_{i=1}^c r_i=r.
\] Thus, $I_G$ has maximal regularity if and only if each $I_{G_i}$ has maximal regularity.

 We may prove now the main statement of this section.

\begin{Theorem}\label{maxreg}
Let $G$ be a  closed graph on the vertex set $[n]$ with the maximal cliques $F_{1},\ldots,F_{r}$, where $F_{i}=[a_{i},b_{i}]$ for $1\leq i\leq r$, 
$1=a_{1} < a_{2} < \cdots < a_{r} < b_{r}=n$.
Then $\reg(S/I_{G})=r$ if and only if $F_{i}\cap F_{i+1}\cap F_{i+2}=\emptyset$ for $1\leq i\leq r-2$.
\end{Theorem}

\begin{proof} By the above discussion, we may reduce the proof of the statement to the connected case. Therefore, in the proof we assume that $G$ is connected.
We know that $H_{S/I_G}(t)=H_{S/\ini_{\rev}(I_G)}(t)=\frac{P(t)}{(1-t)^2}$ with $P(t)\in\mathbb{Z}[t]$. We have  $\ini_{\rev}(I_{G})=\ini_{\rev}(I_{F_{1}}) + \cdots + \ini_{\rev}(I_{F_{r}})=(x_{2},\ldots,x_{b_{1}})^2+ \cdots + (x_{a_{r}+1},\ldots,x_{n})^2$. As $x_{1},$ $x_{n+1}$ is a  regular sequence  on $S/\ini_{\rev}(I_{G})$, we get \[P(t)= H_{S/(\ini_{\rev}(I_{G}),x_{1},x_{n+1})}(t)=h_{0}+h_{1}t+\cdots+h_{r}t^r,\]
where $h_{i}=\dim (S/(\ini_{\rev}(I_{G}),x_{1},x_{n+1}))_{i}$ for $0\leq i \leq r$.

First suppose that $\reg(S/I_G)=r$, that is, $h_{r}\neq0$. This implies that there exists a monomial of degree $r$, say $w,$ which does not belong to $(\ini_{\rev}(I_{G}),x_{1},x_{n+1})$. Observe that any monomial that does not belong to $(\ini_{\rev}(I_{G}),x_{1},x_{n+1})$ is a square free monomial in the variables $x_{2},\ldots,x_{n}$. Let $w=x_{j_{1}},\ldots,x_{j_{r}}\notin \ini_{\rev}(I_{G})$ with $2\leq j_{1}<\cdots<j_{r}\leq n$. Then we must have $j_{1}\in F_{1}, j_{2}\in F_{2}\backslash F_{1},\ldots,j_{r}\in F_{r}\backslash F_{r-1}$. Assume there exists $i$ with $F_{i}\cap F_{i+1}\cap F_{i+2}\neq \emptyset$, where $F_{i}=[a_{i},b_{i}],F_{i+1}=[a_{i+1},b_{i+1}],$ and $F_{i+2}=[a_{i+2},b_{i+2}]$. Since $F_{i}\cap F_{i+1}\cap F_{i+2}\neq \emptyset$, we observe that $a_{i}<a_{i+1}<a_{i+2}<b_{i}<b_{i+1}<b_{i+2}$ which shows that $F_{i+1}\backslash F_{i} \subset F_{i+2}$. This implies that $j_{i+1},j_{i+2}\in F_{i+2}$, that is, $x_{j_{i+1}} x_{j_{i+2}}\in \ini_{\rev}(I_{G})$, which is a contradiction to the 
choice of $w$.

Conversely, suppose that $F_{i}\cap F_{i+1}\cap F_{i+2}=\emptyset$ for all $i$. In other words, we have $a_{i+1}\geq b_i+1$ for 
$1\leq i\leq r-2.$ Then, it is easily seen that the monomial $w=x_2 x_{b_1+1}\cdots x_{b_{r-2}+1}x_n$ has degree $r$ and it does not belong to $(\ini_{\rev}(I_{G}),x_{1},x_{n+1})$. Therefore, $\reg(S/ I_{G})=r$.
\end{proof}

\section{Gorenstein binomial edge ideals associated with scrolls}

In this section we characterize the closed graphs $G$ with the property that $I_G$ is a Gorenstein ideal, that is, $S/I_G$ is Gorenstein.

We first consider the case when $G$ is connected. We note that if $\Delta(G)$ consists of a single clique, that is, $G$ is the complete graph $K_n$, then $I_G$ 
has a linear resolution. Moreover, one may easily derive that $\beta_{n-1}(S/I_G)=n-1$, hence, unless $G=K_2$, $S/I_G$ is not Gorenstein. Therefore, in what 
follows we consider that $G$ has at least 2 cliques.

\begin{Lemma}\label{reg_is_r}
 Let G be a connected closed graph with $r\geq2$ maximal cliques $F_1,\dots,F_r$ with $F_i=[a_i,b_i]$ for $1\leq i\leq r$, $1=a_{1} < a_{2} < \cdots < a_{r} < b_{r}=n$. If $I_G$ is Gorenstein and reg$(S/I_G)=r$, then the following equalities hold:  \[a_2=2, a_{i+2}=b_i+1 {\rm ~for~} 1\leq i\leq r-2, b_{r-1}=n-1.\]
\end{Lemma}

\begin{proof}
By Theorem \ref{maxreg}, we have $F_i\cap F_{i+2}=\emptyset$ for $1\leq i\leq r-2$;
 in other words, $b_i<a_{i+2}$ for $1\leq i\leq r-2$. 
 
 As before, let $P(t)=h_0+h_1t+\cdots+h_rt^r$ be the numerator of the Hilbert series $H_{S/I_G}(t)=H_{S/\ini_{\rev}(I_G)}(t)$. Recall from the proof of Theorem \ref{maxreg} that $P(t)=H_{S/(\ini_{\rev}(I_G),x_1,x_{n+1})}(t)$. 
 
Since $I_G$ is Gorenstein, the $h$-vector $(h_0,\dots,h_s)$ is symmetric. Therefore, since $h_0=1$, the leading coefficient $h_r$ of $P(t)$ must be equal to 1, as well. This means that the last non-zero component of $S/(\ini_{\rev}(I_G),x_1,x_{n+1})$ has dimension 1 as a vector space over $K$. 
 In other words, there must be exactly one 
 squarefree monomial of degree $r$ in the variables $x_2,\dots,x_n$ which does not belong to $\ini_{\rev}(I_G)$. One easily observes that the monomial 
 $w=x_2x_{b_1+1}x_{b_2+1}\cdots x_{b_{r-2}+1}x_n$ does not belong to $\ini_{\rev}(I_G)$. Now we will show that the above equalities must hold in order to not have another squarefree monomial $w^\prime$ of degree $r$ such that $w^\prime \notin\ini_{\rev}(I_G)$.
 
 Let us first assume that $a_2>2$. Then we find the monomial $x_{a_2}x_{b_1+1}x_{b_2+1}\cdots x_{b_{r-2}+1}x_n$ which does not belong to $\ini_{\rev}(I_G)$. 
 Similarly, if $b_{r-1}<n-1$, then we find the monomial $x_2x_{b_1+1}x_{b_2+1}\cdots x_{b_{r-2}+1}x_{n-1}\notin \ini_{\rev}(I_G)$. 
 Finally, if there exists $1\leq i\leq r-2$ such that $a_{i+2}>b_i+1$, then the monomial 
 $\frac{w}{x_{b_i+1}}x_{b_i+2}=x_2x_{b_1+1}\cdots x_{b_{i-1}+1}x_{b_{i}+2}x_{b_{i+1}+1}\cdots x_{b_{r-2}+1}x_n$ does not belong to $\ini_{\rev}(I_G)$. 
\end{proof}

We observe that $x_1,x_{n+1}$ is a regular sequence on $S/I_G$. Indeed, one easily sees that $x_1$ is regular on $S/\ini_{\rev}(I_G)$, and hence on $S/I_G$, and $(\ini_{\rev}(I_G),x_1)=\ini_{\rev}(I_G, x_1)$. As $x_{n+1}$ is regular on $S/\ini_{\rev}(I_G, x_1)$, it follows that $x_{n+1}$ is regular on $S/(I_G, x_1)$, as well. 

Therefore, the ring $S/I_G$ is Gorenstein if and only if $S/(I_G, x_1,x_{n+1})$ is Gorenstein. On the other hand $S/(I_G, x_1,x_{n+1})\cong \bar{S}/\bar{I}_G$, where $\bar{S}=K[x_2,\dots,x_n]$ and $\bar{I}_G=I_G$ mod$(x_1,x_{n+1})$. 
 
 We also observe that $\bar{S}/\bar{I}_G$ is a zero-dimensional ring. Thus, $\bar{S}/\bar{I}_G$ is Gorenstein if and only if the socle of $\bar{S}/\bar{I}_G$ has dimension 1 as a $K$-vector space \cite[Proposition 21.5]{Ei2}. Therefore, $S/I_G$ is Gorenstein if and only if dim$_K(\bar{I}_G : \mathfrak{m}/\bar{I}_G)=1$, where $\mathfrak{m}=(x_2,\dots,x_n)$.

The next theorem is the core of this section.

\begin{Theorem}\label{mainthm}
 Let G be a connected closed graph with $r\geq2$ maximal cliques $F_1,\dots,F_r$ with $F_i=[a_i,b_i]$ for $1\leq i\leq r$, $1=a_{1} < a_{2} < \cdots < a_{r} < b_{r}=n$. Then the following statements are 
 equivalent:
 \begin{itemize}
  \item[(a)] $I_G$ is a Gorenstein ideal;
  \item[(b)] The following equalities hold:  $a_2=2, a_{i+2}=b_i+1$ for $1\leq i\leq r-2$, $b_{r-1}=n-1$.
 \end{itemize}
\end{Theorem}
\begin{proof}
 $(a)\Rightarrow(b).$ Suppose that $I_G$ is a Gorenstein ideal. By the above observations, we have dim$_K(\bar{I}_G : \mathfrak{m}/\bar{I}_G)=1$, where $\bar{S}=K[x_2,\dots,x_n]$, $\bar{I}_G=I_G$ mod$(x_1,x_{n+1})$, and $\mathfrak{m}=(x_2,\dots,x_n)$.
 
 We easily see that the reduced Gr\"{o}bner basis of $\bar{I}_G$ with respect to the reverse lexicographic order is obtained from the reduced Gr\"{o}bner basis of $I_G$ by moding out $x_1$ and $x_{n+1}$. Therefore, 
 \[\ini_{\rev}(\bar{I}_G)=(x_2,\dots,x_{b_1})^2+(x_{a_2+1},\dots,x_{b_2})^2+\cdots+(x_{a_r+1},\dots,x_n)^2.\] Since $S/I_G$ is Gorenstein, recall from the proof of Lemma \ref{reg_is_r} that there must be exactly one squarefree monomial in the maximal degree $s=\reg(S/I_G)=\reg(\bar{S}/\bar{I}_G)\leq r$ in the variables $x_2,\dots,x_n$ which does not belong to $\ini_{\rev}(\bar{I}_G)$. In other words, $h_s=1$, where $h_s$ is the last component of the $h$-vector of $\bar{S}/\ini_{\rev}(\bar{I}_G)$, which coincides with the $h$-vector of $S/\ini_{\rev}(I_G)$.
 
 We will show that if $h_s=1$ and $s<r$, then dim$_K(\bar{I}_G : \mathfrak{m}/\bar{I}_G)>1$. Once we prove this, since $\bar{S}/\bar{I}_G$ is Gorenstein, it follows that we must have $s=r$ and the proof of $(a)\Rightarrow(b)$ is completed by Lemma \ref{reg_is_r}.
 
 So, suppose that $h_s=1$ and $s<r$. Let $w=x_{j_1}x_{j_2}\cdots x_{j_s}$ be the unique squarefree monomial of degree $s$ with $w\notin \ini_{\rev}(\bar{I}_G)$.
 
 The uniqueness of $w$ immediately implies that $j_1=2$ and $j_s=n$. Indeed, if for example $j_1>2$, then $w'=x_2x_{j_2}\cdots x_{j_s}\notin \ini_{\rev}(\bar{I}_G)$ and $w'\neq w$, contradiction. Thus, $w=x_2x_{j_2}\cdots x_{j_{s-1}}x_n$. Moreover, we must have $j_2\geq b_1+1$ and, by the uniqueness of $w$, we get $j_2= b_1+1$.
 
 Again by the uniqueness of $w$, it follows that $a_2=2$ and $b_{r-1}=n-1$. Indeed, if, for example $a_2>2$, then $w'=x_3x_{j_2}\cdots x_{j_{s-1}}x_n\notin \ini_{\rev}(\bar{I}_G)$ and $w'\neq w$, which is again a contradiction.
 
  On the other hand, we observe that $w$ must "cover" every set $A_i=\{x_{a_{i}+1},x_{a_{i}+2},\dots,x_{b_i}\}$ for $1\leq i\leq r$ in the sense that for each $i$ with $1\leq i\leq r$ there exists a variable $x_j\in A_i$ such that $x_j\mid w$. Indeed, let us assume that there exists $i$ with $1\leq i\leq r$ such that for all $x_j\in A_i$, $x_j\nmid w$. Then we may find an integer $q$ such that $j_q<a_i+1\leq b_i<j_{q+1}$. It follows that $w'=\frac{w}{x_{j_q}}x_{a_{i}+1}\notin\ini_{\rev}(\bar{I}_G)$ and $\deg(w')=\deg(w)$, a contradiction to the uniqueness of $w$.
  
  In addition, the uniqueness of $w$ implies that for all $i$ with $1\leq i\leq r-2$ we must have $a_{i+2}\leq b_i+1$. Indeed, let us assume that there exists an integer $k$ such that $a_{k+2}> b_k+1$ and set $i=\min\{k\mid a_{k+2}> b_k+1\}$. There exists a unique $q$ such that $j_q\leq b_i$ and $j_{q+1}>b_i$. If $j_{q+1}= b_i+1$, then $w'=\frac{w}{x_{j_{q+1}}}x_{b_i+2}\notin \ini_{\rev}(\bar{I}_G)$ and $\deg(w')=\deg(w)$. If $j_{q+1}>b_i+1$, then $w'=\frac{w}{x_{j_{q+1}}}x_{b_i+1}\notin \ini_{\rev}(\bar{I}_G)$ and $\deg(w')=\deg(w)$.
  
  Consequently, in what follows we assume that $w$ covers every set $A_i$, $i=1,\dots,r$, and $a_{i+2}\leq b_i+1$, for $1\leq i\leq r-2$.
 
 For $1\leq k\leq s$, we set $i(k)= {\rm min}\{i\mid j_k\in A_i\}$, that is, $A_{i(k)}$ is the first set $A_i$ containing $j_k$ and $a_{i(k)}, b_{i(k)}$ are the endpoints of the clique $F_{i(k)}$. For example, as $j_1=2$, we have $i(1)=1$ and, since $j_2=b_1+1$, we have $i(2)=2$.
 
 Since $w$ covers each set $A_i$ and by its uniqueness, we derive that for $3\leq k\leq s$ the following conditions must be fulfilled: 
\[j_k= {\rm max}\{b_i\mid x_{j_{k-1}}\in A_i\}+1~{\rm and}~a_{i(k)}=j_{k-1}=b_{i(k-1)-1}+1.\]
 
 In particular, $a_r=b_{i(s-1)-1}+1.$
 
 Note that $\bar{I}_G : \mathfrak{m}/\bar{I}_G$ contains in its $K$-basis the monomial $w$ of degree $s=$ reg$(\bar{S}/\bar{I}_G)< r$. We will show that under the above conditions, which are imposed by the uniqueness of $w$, we may find another polynomial $f$ of degree $s-1$ in the $K$-basis of $\bar{I}_G : \mathfrak{m}/\bar{I}_G$, hence obtaining dim$_K(\bar{I}_G : \mathfrak{m}/\bar{I}_G)>1$, as claimed.
 
 Recall that $w=x_2x_{j_2}\cdots x_{j_{s-1}}x_n$ with $j_{k}=b_{i(k)-1}+1$ for $2\leq k\leq s-1$ and $a_{i(k)}=b_{i(k-1)-1}+1$ for $2\leq k\leq s$. 
 
 Note that, since $s=\reg(\bar{S}/\bar{I}_G)=\reg(S/I_G)<r$, there must exist three consecutive cliques with nonempty intersection. We may assume without loss of generality that the intersection of the first 3 cliques is nonempty, that is, $i(3)>3.$
 
 We consider the following binomial of degree $s-1$: \[f=(x_3x_{j_3-1}-x_2x_{j_3})x_{j_4}\cdots x_{j_{s-1}}x_n.\]

Obviously, $\ini_{\rev}(f)=x_3x_{j_3-1}x_{j_4}\cdots x_{j_{s-1}}x_n$. We claim that $\ini_{\rev}(f)\notin \ini_{\rev}(\bar{I}_G)$. Indeed, if $\ini_{\rev}(f)\in \ini_{\rev}(\bar{I}_G)$, then we should have either $x_3x_{j_3-1}\in \ini_{\rev}(\bar{I}_G)$ or $x_{j_3-1}x_{j_4}\in \ini_{\rev}(\bar{I}_G)$. But, as $i(3)>3$, $j_3-1=b_{i(3)-1}>b_2$, thus $x_3x_{j_3-1}\notin \ini_{\rev}(\bar{I}_G)$. If $x_{j_3-1}x_{j_4}\in \ini_{\rev}(\bar{I}_G)$, then $x_{j_3}x_{j_4}\in \ini_{\rev}(\bar{I}_G)$, impossible. Thus, $\ini_{\rev}(f)\notin \ini_{\rev}(\bar{I}_G)$, which implies that $f\notin\bar{I}_G$. 

In order to obtain dim$_K(\bar{I}_G : \mathfrak{m}/\bar{I}_G)>1$, it remains to show that $x_lf\in\bar{I}_G$ for $2\leq l\leq n$. This will complete the proof of the theorem. (See Example \ref{atob} for an illustration of the following technical procedure.)
 
 \underline{Case 1:} If $2\leq l\leq b_1-1$, then $x_3x_l\equiv x_2x_{l+1}\in \bar{I}_G$ and $x_2x_l\in \bar{I}_G$, thus $x_lf\in\bar{I}_G$.
 
 \underline{Case 2:} If $l=b_1$, then \[x_3x_{b_1}x_{j_3-1}\equiv x_2x_{b_1+1}x_{j_3-1}\equiv x_2x_{b_1}x_{j_3}\in \bar{I}_G\] and $x_2x_{b_1}\in \bar{I}_G$, thus $x_{b_1}f\in\bar{I}_G$.
  
 \underline{Case 3:} If $l=b_1+1$, then \[x_3x_{b_1+1}x_{j_3-1}\equiv x_3x_{b_1}x_{j_3}\equiv x_2x_{b_1+1}x_{j_3},\] the first congruence holding because $b_1\in F_{i(3)-1}$. Therefore, $x_{b_1+1}(x_3x_{j_3-1}-x_2x_{j_3})\in\bar{I}_G$, and thus $x_{b_1+1}f\in\bar{I}_G$.
 
 \underline{Case 4:} If $l=b_1+2$, then \[x_3x_{b_1+2}x_{j_3-1}\equiv x_3x_{b_1+1}x_{j_3}\equiv x_2x_{b_1+2}x_{j_3}.\] As in the previous case it follows that $x_{b_1+2}(x_3x_{j_3-1}-x_2x_{j_3})\in\bar{I}_G$, and thus $x_{b_1+2}f\in\bar{I}_G$.
 
 \underline{Case 5:} If $b_1+3\leq l\leq j_3$, since $a_{i(3)}=b_1+1$, we get:
 \begin{align*}
        &x_3x_lx_{j_3-1}x_{j_4}\cdots x_{j_{s-1}}x_n\equiv\\
 &\equiv x_3x_{l-1}x_{j_3}x_{j_4}\cdots x_{j_{s-1}}x_n\equiv\\
 &\equiv x_3x_{l-2}x_{j_3+1}x_{j_4}\cdots x_{j_{s-1}}x_n\equiv\\
 &\equiv x_3x_{l-2}x_{j_3}x_{j_4+1}\cdots x_{j_{s-1}}x_n\equiv\\
 &~~~~~~\vdots\\
 &\equiv x_3x_{l-2}x_{j_3}x_{j_4}\cdots x_{j_{s-1}+1}x_n\equiv\\
 &\equiv0~({\rm mod} ~\bar{I}_G),
 \end{align*}
 the last congruency holding because $x_{j_{s-1}+1}x_n\in\bar{I}_G$. 
 
 One shows similarily that $x_2x_lx_{j_3}x_{j_4}\cdots x_{j_{s-1}}x_n\in\bar{I}_G$. Thus $x_lf\in\bar{I}_G$.
  
 \underline{Case 6:} If $j_{k-1} < l \leq j_k$, where $4\leq k\leq s$, then
  \begin{align*}
  x_lf&\equiv (x_3x_{j_3-1}-x_2x_{j_3})x_{j_4}\cdots x_lx_{j_k}x_{j_{k+1}}\cdots x_{j_{s-1}}x_n\equiv\\
      &\equiv (x_3x_{j_3-1}-x_2x_{j_3})x_{j_4}\cdots x_{l-1}x_{j_k+1}x_{j_{k+1}}\cdots x_{j_{s-1}}x_n\equiv\\
      &\equiv (x_3x_{j_3-1}-x_2x_{j_3})x_{j_4}\cdots x_{l-1}x_{j_k}x_{j_{k+1}+1}\cdots x_{j_{s-1}}x_n\equiv\\
      &~~~~~~\vdots\\
      &\equiv (x_3x_{j_3-1}-x_2x_{j_3})x_{j_4}\cdots x_{j_{s-2}}x_{j_{s-1}+1}x_n\equiv\\
      &\equiv0~({\rm mod} ~\bar{I}_G).
 \end{align*}
 $(b)\Rightarrow(a).$ Let us assume that $F_1,\dots,F_r$ satisfy the numerical conditions of (b). We will show that dim$_K(\bar{I}_G : \mathfrak{m}/\bar{I}_G)=1$.
 
 The hypothesis on $G$ ensures the existence of a unique monomial $w$ of degree $r$ such that $w\notin\bar{I}_G$ and $\mathfrak{m}w\subseteq\bar{I}_G$, namely, $w=x_2x_{b_1+1}x_{b_2+1}\cdots x_{b_{r-2}+1}x_n$. Therefore, it remains to prove the following claim: 
 \begin{center}
  if $f\in \bar{I}_G : \mathfrak{m}$ and deg$(f)\leq r-1$, then $f\in \bar{I}_G$.
 \end{center}
 
 We prove this claim by contradiction. Let us assume that there exists a homogeneous polynomial $f\in \bar{I}_G : \mathfrak{m}$ with deg$(f)\leq r-1$ such that 
 $f\notin \bar{I}_G$. 
 
Let $f=c_1u_1+c_2u_2+\cdots+c_mu_m$, where $u_1>_{\rev} u_2>_{\rev}\cdots>_{\rev} u_m$ and $c_i\in K\backslash\{0\}$ for $1\leq i\leq m$. By reducing $f$ modulo $\bar{I}_G$, we may assume that no monomial in the support of $f$ belongs to $\ini_{\rev}(\bar{I}_G)$, and hence, each $u_i$ is a squarefree monomial in the variables $x_2,x_3,\dots,x_n$.

In order to reach the contradiction, we need to find a variable $x_k$, with $2\leq k\leq n$, such that $x_kf$ is congruent modulo $\bar{I}_G$ to some polynomial $g$ which does not belong to $\bar{I}_G$. We use an inductive procedure to find the appropriate variable $x_k$. (See Example \ref{btoa} for an illustration of this procedure.)

 \underline{Step 1:} Let us first assume that there exists some monomial $u_q\in\Supp(f)$ such that $u_q$ does not cover the set $A_1=\{x_2,x_3,\dots,x_{b_1}\}$ and let $i=\min\{q\mid u_q \text{ does not cover } A_1\}$. Hence, for all $x_l\in A_1$, $x_l\nmid u_i$, and $i$ is the least index with this property. We have \[x_2f=c_1x_2u_1+\cdots+c_{i-1}x_2u_{i-1}+c_{i}x_2u_{i}+\cdots c_mx_2u_m.\] Since for every $j\leq i-1$ there exists $x_l\in A_1$ such that $x_l\mid u_j$, it follows that $x_2u_j\in\bar{I}_G$, because $x_2x_l\in\bar{I}_G$. Hence, $x_2f\equiv g=c_{i}x_2u_{i}+\cdots c_mx_2u_m$ (mod $\bar{I}_G$). Now, since $x_2u_i=\ini_{\rev}(g)\notin\ini_{\rev}(\bar{I}_G)$, it follows that $g\notin\bar{I}_G$, thus $x_2f\notin\bar{I}_G$. 
 
 Hence, in what follows, we assume that every $u_i\in\Supp(f)$ covers $A_1$.

 \underline{Step 2:} Let us first assume that there exists $u_i\in\Supp(f)$ such that $x_l\mid u_i$ for some $l$ with $3\leq l\leq b_1$. 
 
 First, we suppose that we have in the support of $f$ the monomials $u_{q_1}>_{\rev}\cdots>_{\rev} u_{q_s}$ which are divisible by $x_{b_1}$. Then 
 $x_3f\equiv c_{q_1}x_3u_{q_1}+\cdots+c_{q_s}x_3u_{q_s}$, since $x_3x_l\in\bar{I}_G$, for $2\leq l\leq b_1-1$, that is
 \[x_3f\equiv x_3x_{b_1}\bigg(c_{q_1}\frac{u_{q_1}}{x_{b_1}}+\cdots+c_{q_s}\frac{u_{q_s}}{x_{b_1}}\bigg)\equiv g= x_2x_{b_1+1}\bigg(c_{q_1}\frac{u_{q_1}}{x_{b_1}}+\cdots+c_{q_s}\frac{u_{q_s}}{x_{b_1}}\bigg).\] 
 Since $\ini_{\rev}(g)=x_2x_{b_1+1}\frac{u_{q_1}}{x_{b_1}}$ does not belong to $\ini_{\rev}(\bar{I}_G)$, we have $g\notin\bar{I}_G$, and thus $x_3f\notin\bar{I}_G$.
 
 Next, we suppose that $\max\{l\in A_1\mid x_l \text{ divides } u_i \text{ for some } 1\leq i\leq m\}=b_1-1$. If $u_{q_1}>_{\rev}\cdots>_{\rev} u_{q_s}$ are the monomials of the support of $f$ which are divisible by $x_{b_1-1}$, then we consider $x_4f$ and get 
 \[x_4f\equiv x_4x_{b_1-1}\bigg(c_{q_1}\frac{u_{q_1}}{x_{b_1-1}}+\cdots+c_{q_s}\frac{u_{q_s}}{x_{b_1-1}}\bigg) ~({\rm mod} ~\bar{I}_G).\]
 But $x_4x_{b_1-1}\equiv x_3x_{b_1}\equiv x_2x_{b_1+1} ~({\rm mod} ~\bar{I}_G)$, thus 
 \[x_4f\equiv g= x_2x_{b_1+1}\bigg(c_{q_1}\frac{u_{q_1}}{x_{b_1-1}}+\cdots+c_{q_s}\frac{u_{q_s}}{x_{b_1-1}}\bigg) ~({\rm mod} ~\bar{I}_G).\]
 As $\ini_{\rev}(g)=x_2x_{b_1+1}\frac{u_{q_1}}{x_{b_1-1}}\notin\ini_{\rev}(\bar{I}_G)$, we have $g\notin\bar{I}_G$, and thus $x_4f\notin\bar{I}_G$. Contradiction.
 
 By repeating this procedure for $\max\{l\in A_1\mid x_l \text{ divides } u_i \text{ for some } 1\leq i\leq m\}=b_1-2$, $b_1-3,\dots,4,3$, and by using the congruences $x_2x_{b_1+1}\equiv x_3x_{b_1}\equiv x_4x_{b_1-1}\equiv x_5x_{b_1-2}\equiv\cdots$, we may find, in each case, a suitable variable $x_k$ such that $x_kf\notin \bar{I}_G$. 
 
 Therefore, we conclude that $x_2\mid u_i$ for all $u_i\in\Supp(f)$.

\underline{Step 3:} By induction on $j$, we may assume that $x_2x_{b_1+1}\cdots x_{b_{j-2}+1}\mid u_i$ for all $u_i\in\Supp(f)$.

Let us now first consider the case when there exists some monomial $u_q\in\Supp(f)$ which does not cover $A_j=\{x_{a_j+1},\dots,x_{b_j}\}$. Let $i=\min\{q\mid u_q \text{ does not cover } A_j\}$. We will show that $x_{b_{j-1}+1}f\notin \bar{I}_G$.

For $k\leq i-1$, $u_k$ is of the form $x_2x_{b_1+1}\cdots x_{b_{j-2}+1}x_{l}v$ for some monomial $v$ and some variable $x_l\in A_j$, with $l\geq b_{j-1}+1$.

We have
\begin{align*}
x_{b_{j-1}+1}u_k&=x_2x_{b_1+1}\cdots x_{b_{j-2}+1}x_{b_{j-1}+1}x_{l}v\equiv\\
              &\equiv x_2x_{b_1+1}\cdots x_{b_{j-2}+1}x_{b_{j-1}}x_{l+1}v\equiv\\
              &\equiv x_2x_{b_1+1}\cdots x_{b_{j-2}}x_{b_{j-1}+1}x_{l+1}v\equiv\\
              &\vdots\\
              &\equiv x_2x_{b_1}x_{b_2+1}\cdots x_{b_{j-1}+1}x_{l+1}v\equiv\\
              &\equiv0~(\mod~ \bar{I}_G),
\end{align*}
the last congruence holding because $x_2x_{b_1}\in\bar{I}_G$. 

It follows that \[x_{b_{j-1}+1}f\equiv g=c_ix_{b_{j-1}+1}u_i+\cdots+c_mx_{b_{j-1}+1}u_m.\] 

By our assumption on $u_i$, we have that $\ini_{\rev}(g)=c_ix_{b_{j-1}+1}u_i\notin\ini_{\rev}(\bar{I}_G)$, thus $g\notin\bar{I}_G$ and $x_{b_{j-1}+1}f\notin\bar{I}_G$. Contradiction.

Finally, we consider the case when each $u_q\in\Supp(f)$ covers $A_j=\{x_{a_j+1},\dots,x_{b_j}\}$. 

In this case, either $x_{b_{j-1}+1}$ divides each $u_q\in\Supp(f)$, which takes us back to the beginning of Step 3 with $j+1$ instead of $j$ (a procedure which has to terminate), or there is a monomial in the support of $f$ which is divisible by some variable $x_l$, where $b_{j-1}+2\leq l\leq b_j$.

In the latter case we proceed as in Step 2, by considering \[\max\{l\in A_j\cap A_{j+1}\mid x_l \text{ divides some } u_i \in\Supp(f)\}.\]

We illustrate the procedure when the above maximum is equal to $b_j$. Let $u_{q_1}>_{\rev}\cdots>_{\rev} u_{q_s}$ be the monomials in the support of $f$ which are divisible by $x_{b_j}$. The other monomials in the support of $f$ (if any) must be divisible by \[x_2x_{b_1+1}\cdots x_{b_{j-2}+1}x_{b_{j-1}+1}.\]

We will show that $x_{b_{j-1}+2}f\notin\bar{I}_G$. First we observe that for $1\leq i\leq s$, \[x_{b_{j-1}+2}u_{q_i}=x_{b_{j-1}+2}x_{b_j}\frac{u_{q_i}}{x_{b_j}}\equiv x_{b_{j-1}+1}x_{b_j+1}\frac{u_{q_i}}{x_{b_j}} ~(\mod~ \bar{I}_G).\]
  Note that the latter monomial is not in $\ini_{\rev}(\bar{I}_G)$. If we show that $x_{b_{j-1}+2}u\equiv0~(\mod~ \bar{I}_G)$ for any other monomial $u\in\Supp(f)$, it will follow that \[x_{b_{j-1}+2}f\equiv g= \sum_{i=1}^sx_{b_{j-1}+1}x_{b_j+1}\bigg(c_i\frac{u_{q_i}}{x_{b_j}}\bigg)~(\mod~ \bar{I}_G),\]
and, since $\ini_{\rev}(g)\notin\ini_{\rev}(\bar{I}_G)$, we have $g\notin\bar{I}_G$, and thus $x_{b_{j-1}+2}f\notin\bar{I}_G$.

Let $u$ be of the form $x_2x_{b_1+1}\cdots x_{b_{j-2}+1}x_{l}v$ for some monomial $v$ and some variable $x_l$ with $b_{j-1}+2\leq l< b_j$. Then
\begin{align*}
x_{b_{j-1}+2}u&=x_2x_{b_1+1}\cdots x_{b_{j-2}+1}x_{b_{j-1}+2}x_{l}v\equiv\\
              &\equiv x_2x_{b_1+1}\cdots x_{b_{j-2}+1}x_{b_{j-1}+1}x_{l+1}v\equiv\\
              &\equiv x_2x_{b_1+1}\cdots x_{b_{j-2}+1}x_{b_{j-1}}x_{l+2}v\equiv\\
              &\equiv x_2x_{b_1+1}\cdots x_{b_{j-2}}x_{b_{j-1}+1}x_{l+2}v\equiv\\
              &\vdots\\
              &\equiv x_2x_{b_1}x_{b_2+1}\cdots x_{b_{j-1}+1}x_{l+2}v\equiv\\
              &\equiv0~(\mod~ \bar{I}_G),
\end{align*}
This completes the proof of the theorem.
\end{proof}

We now give two examples which illustrate the technical procedure of the first part of the above proof. The first example illustrates the case which is considered in the proof, that is, when the intersection of the first 3 cliques in nonempty, whereas the second example illustrates the case when 3 consecutive cliques with nonempty intersection occur later, not at the beginning. 

From these 2 examples it is clear how the polynomial $f$ in the first part of the proof is picked in general.

\begin{Example}\label{atob}{\rm
Consider the graph $G$ with the cliques $F_1=[1,5], F_2=[2,6], F_3=[3,8], F_4=[4,9], F_5=[6,10], F_6=[7,12], F_7=[8,13], F_8=[10,14]$. Using the notation of the proof of Theorem \ref{mainthm}, we observe that $x_{j_1}x_{j_2}x_{j_3}x_{j_4}=x_2x_6x_{10}x_{14}$ is the only monomial of degree 4 in $\bar{I}_G : \mathfrak{m}/\bar{I}_G$. We also note that $b_{i(3)}=10$, and hence $b_{i(3)-1}=9$. Therefore, we consider
\[f=(x_3x_9-x_2x_{10})x_{14}.\]
We see that $\ini_{\rev}(f)=x_3x_{9}x_{14}$ does not belong to $\ini_{\rev}(\bar{I}_G)$. Hence, indeed, $f\notin\bar{I}_G$. In order to show that $f\in\bar{I}_G : \mathfrak{m}/\bar{I}_G$, hence obtaining dim$_K(\bar{I}_G : \mathfrak{m}/\bar{I}_G)>1$, it remains to show that $x_lf\in\bar{I}_G$ for $2\leq l\leq 14$. 

 \underline{Case 1:} If $2\leq l\leq b_1-1=4$, then $x_3x_l\equiv x_2x_{l+1}\in \bar{I}_G$ and $x_2x_l\in \bar{I}_G$, thus $x_lf\in\bar{I}_G$.
 
 \underline{Case 2:} If $l=b_1=5$, then $x_3x_{5}x_{9}\equiv x_2x_{6}x_{9}\equiv x_2x_{5}x_{10}\in \bar{I}_G$ and $x_2x_{5}\in \bar{I}_G$, thus $x_{5}f\in\bar{I}_G$.
  
 \underline{Case 3:} If $l=b_1+1=6$, then we have $x_3x_{6}x_{9}\equiv x_3x_{5}x_{10}\equiv x_2x_{6}x_{10}$. It follows that $x_{6}(x_3x_{9}-x_2x_{10})\in\bar{I}_G$, and thus $x_{6}f\in\bar{I}_G$.
 
 \underline{Case 4:} If $l=b_1+2=7$, then $x_3x_{7}x_{9}\equiv x_3x_{6}x_{10}\equiv x_2x_{7}x_{10}$. As in the previous case it follows that $x_{7}f\in\bar{I}_G$.
 
 \underline{Case 5:} If $8=b_1+3\leq l\leq j_3=10$, say $l=8$, we have \[x_3x_lx_{9}x_{14}=x_3x_8x_{9}x_{14}\equiv x_3x_{7}x_{10}x_{14}\equiv x_3x_{6}x_{11}x_{14}\equiv0~({\rm mod} ~\bar{I}_G).\]
 Similarily, \[x_2x_lx_{10}x_{14}=x_2x_8x_{10}x_{14}\equiv x_2x_{7}x_{11}x_{14}\equiv 0~({\rm mod} ~\bar{I}_G).\] Thus $x_8f\in\bar{I}_G$.
  
 \underline{Case 6:} If $10=j_3 < l \leq j_4=14$, then $x_lf=(x_3x_9-x_2x_{10})x_lx_{14}\equiv0~({\rm mod} ~\bar{I}_G)$, because $x_lx_{14}\in\bar{I}_G$.
}
\end{Example}

\begin{Example}\label{atob2}{\rm
Consider the graph $G$ with the cliques $F_1=[1,4], F_2=[2,5], F_3=[5,9], F_4=[6,10], F_5=[7,12], F_6=[8,13], F_7=[10,14], F_8=[14,15]$. We observe that $x_2x_5x_{6}x_{10}x_{14}x_{15}$ is the only monomial of degree 5 in $\bar{I}_G : \mathfrak{m}/\bar{I}_G$. We consider
\[f=x_2x_5(x_7x_{13}-x_6x_{14})x_{15}.\]
We see that $\ini_{\rev}(f)=x_2x_5x_{7}x_{13}x_{15}$ does not belong to $\ini_{\rev}(\bar{I}_G)$. Hence, $f\notin\bar{I}_G$. In order to show that $f\in\bar{I}_G : \mathfrak{m}/\bar{I}_G$, hence obtaining dim$_K(\bar{I}_G : \mathfrak{m}/\bar{I}_G)>1$, it remains to show that $x_lf\in\bar{I}_G$ for $2\leq l\leq 15$. 

\underline{Case 1:} If $2=j_1\leq l\leq j_2=5$, then $l=5$ is the nontrivial case. We have \[x_5f=x_2x_5^2(x_7x_{13}-x_6x_{14})x_{15}\equiv x_2x_4x_6(x_7x_{13}-x_6x_{14})x_{15}\in \bar{I}_G.\]

\underline{Case 2:} If $6=j_3\leq l\leq b_{i(3)}-1=8$, say $l=8$, then \[x_8x_2x_5x_7x_{13}x_{15}\equiv x_2x_5x_6x_9x_{13}x_{15}\equiv x_2x_5^2x_{10}x_{13x_{15}}\equiv x_2x_4x_6x_{10}x_{13}x_{15}\in \bar{I}_G,\]
and \[x_8x_2x_5x_6x_{14}x_{15}\equiv x_2x_5^2x_9x_{14}x_{15}\equiv x_2x_4x_6x_{9}x_{14}x_{15}\in \bar{I}_G,\]
hence, $x_8f\in\bar{I}_G$. By the same argument, we get $x_6f\in\bar{I}_G$ and $x_7f\in\bar{I}_G$.

\underline{Case 3:} If $9=b_{i(3)}\leq l\leq b_{i(3)}+2=11$, say $l=9$, then \[x_7x_9x_{13}\equiv x_7x_8x_{14}\equiv x_6x_9x_{14}.\]
Therefore, $x_9(x_7x_{13}-x_6x_{14})\in \bar{I}_G$, and hence, $x_9f\in \bar{I}_G$. By the same argument, we get $x_{10}f\in\bar{I}_G$ and $x_{11}f\in\bar{I}_G$.

\underline{Case 4:} If $12=b_{i(3)}+3\leq l\leq j_5=14$, say $l=12$, then \[x_{12}x_2x_5x_7x_{13}x_{15}\equiv x_2x_5x_7x_{11}x_{14}x_{15}\equiv x_2x_5x_{7}x_{10}x_{15}^2\in \bar{I}_G,\]
and \[x_{12}x_2x_5x_6x_{14}x_{15}\equiv x_2x_5x_6x_{11}x_{15}^2\in \bar{I}_G.\]
It follows that $x_{12}f\in\bar{I}_G$. By the same argument, we get $x_{13}f\in\bar{I}_G$ and $x_{14}f\in\bar{I}_G$.

\underline{Case 5:} If $14=j_5< l\leq j_6=15$, then $x_{15}f=x_2x_5(x_7x_{13}-x_6x_{14})x_{15}^2\in\bar{I}_G$.
}
\end{Example}

We now give an example which illustrates the technical procedure of the second part of the above proof.

\begin{Example}\label{btoa}{\rm
Consider the graph $G$ with the cliques $F_1=[1,5], F_2=[2,9], F_3=[6,14]$, $F_4=[10,17], F_5=[15,21], F_6=[18,22]$, and consider the polynomial 
\[f=x_2x_{6}x_{10}x_{15}+ x_2x_{6}x_{11}x_{21}+x_2x_{6}x_{12}x_{21}+x_2x_{6}x_{13}x_{21}+x_2x_{6}x_{14}x_{21}.\] 

We will show that there exists a variable $x_k$ such that $x_kf\notin \bar{I}_G$. Note that we are in the last subcase of Step 3 with $j=3$.

We observe that $\max\{l\in A_3\cap A_{4}\mid x_l \text{ divides some } u_i \in\Supp(f)\}=14=b_3$. Therefore, we multiply $f$ by $x_{b_{j-1}+2}=x_{11}$ and show that $x_{11}f\notin\bar{I}_G$. We have 
\begin{align*}
x_{11}u_1&=x_2x_{6}x_{10}x_{11}x_{15}\equiv x_2x_{6}x_{9}x_{12}x_{15}\equiv x_2x_{5}x_{10}x_{12}x_{15}\equiv 0~(\mod~ \bar{I}_G)\\
x_{11}u_2&=x_2x_{6}x_{11}^2x_{21}\equiv x_2x_{6}x_{10}x_{12}x_{21}\equiv x_2x_{6}x_{9}x_{13}x_{21}\equiv x_2x_{5}x_{10}x_{13}x_{21}\equiv0~(\mod~ \bar{I}_G)\\
x_{11}u_3&=x_2x_{6}x_{11}x_{12}x_{21}\equiv x_2x_{6}x_{10}x_{13}x_{21}\equiv x_2x_{6}x_{9}x_{14}x_{21}\equiv x_2x_{5}x_{10}x_{14}x_{21}\equiv0~(\mod~ \bar{I}_G)\\
x_{11}u_4&=x_2x_{6}x_{11}x_{13}x_{21}\equiv x_2x_{6}x_{10}x_{14}x_{21}\equiv x_2x_{6}x_{9}x_{15}x_{21}\equiv x_2x_{5}x_{10}x_{15}x_{21}\equiv0~(\mod~ \bar{I}_G)
\end{align*}
Finally, $x_{11}u_5=x_2x_{6}x_{11}x_{14}x_{21}\equiv x_2x_{6}x_{10}x_{15}x_{21}~(\mod~ \bar{I}_G)$. Since the latter monomial does not belong to $\ini_{\rev}(\bar{I}_G)$, it is not in $\bar{I}_G$ either. Hence, $x_{11}u_5\notin\bar{I}_G$. It follows that $x_{11}f\notin\bar{I}_G$, because, by the above four congruencies, $x_{11}u_5\equiv x_{11}f~(\mod~ \bar{I}_G)$.
}
\end{Example}

An immediate consequence of Theorem \ref{mainthm} is the following

\begin{Corollary}
If $I_G$ is Gorenstein, then $I_G$ has maximal regularity.
\end{Corollary}

\begin{Example}{\rm Assume that $G$ is connected and has two cliques, say, $F_1=[1,b], F_2=[a,n]$. By Theorem \ref{mainthm}, $I_G$ is Gorenstein if and only if 
$a=2$ and $b=n-1$. So there exists exactly one Gorenstein ideal $I_G$ when $G$ has two cliques. Since reg$(S/I_G)=2$, $I_G$ is extremal Gorenstein (\cite{Sch}, \cite{KSK}), and hence, according to \cite[Theorem B]{Sch}, its Betti numbers are 
\[\beta_{i,i+1}(S/I_G)=\binom {n}{i+1}i-\binom {n-1}{i-1}, ~{\rm for}~ 1\leq i\leq n-1,\] 

\[\beta_{i,i+2}(S/I_G)=0, ~{\rm for}~ 1\leq i\leq n-2,~{\rm and}\]

\[\beta_{n-1,n+1}(S/I_G)=1.\]
 
 }
\end{Example}

The following proposition generalizes the above theorem to all closed graphs.

\begin{Proposition}\label{congor}
 Let $G$ be a closed graph with the connected components $G_1,\dots,G_c$. Then $I_G$ is Gorenstein if and only if $I_{G_i}$ is Gorenstein for all $1\leq i\leq c$.
\end{Proposition}

\begin{proof}
 As in Section 1, let $B_M(s,t)=\sum_{i,j}\beta_{ij}s^it^j$ denote the Betti polynomial of a module $M$. Let $M_i$ be the minimal 
 set of monomial generators of $\ini_{\rev}(I_{G_i})$ for $1\leq i\leq c$. Then $M_i\cap M_j=\emptyset$ for all $i\neq j$. As in the proof of 
 Proposition \ref{bettieq} we derive that Tor$_k(S/I_{G_i}, S/I_{G_j})=0$ for $k>0$ and $i\neq j$, hence 
 \begin{align}  \label{keyequality}
 B_{S/I_G}(s,t)=\prod\limits_{i=1}^{c}B_{S/I_{G_i}}(s,t).
 \end{align}
 Let $r$ = reg$(S/I_G)$. Then $I_G$ is a Gorenstein ideal if and only if $\beta_{n-c,n-c+r}(S/I_G)=1$ and $\beta_{n-c,j}(S/I_G)=0$ for $j\leq n-c+r-1$. Let 
 $V(G_i)=\{n_{i-1}+1,\dots,n_{i-1}+n_i\}$, where $n_0=0$, and let $r_i$ = reg$(S/I_{G_i})$, for $1\leq i\leq c$. Clearly, by equality (\ref{keyequality}), it 
 follows that $\beta_{n-c,n-c+r}(S/I_G)=1$ if and only if, for all $1\leq i\leq c$, $\beta_{n_i-1,n_i-1+r_i}(S/I_{G_i})=1$.
 
 Now, if $I_{G_i}$ is Gorenstein for all $i$, then $\beta_{n_i-1,l}(S/I_{G_i})=0$ for $l<n_i-1+r_i$. By using equality (\ref{keyequality}) again, this implies that 
 $\beta_{n-c,j}(S/I_G)=0$ for $j<n-c+r$, thus $I_G$ is Gorenstein.
 
 For the converse, we argue by contradiction. Let us assume that $I_G$ is Gorenstein and that there exists $1\leq i\leq c$ such that $I_{G_i}$ is not Gorenstein. 
 Since $\beta_{n_i-1,n_i-1+r_i}(S/I_{G_i})=1$, there exists an integer $l<r_i$ such that $\beta_{n_i-1,n_i-1+l}(S/I_{G_i})\geq 1$. By using 
 (\ref{keyequality}), we get: \[\beta_{n-c,n-c+r-r_i+l}(S/I_G)\geq\prod\limits_{j\neq i}\beta_{n_j-1,n_j-1+r_j}(S/I_{G_j})\cdot\beta_{n_i-1,n_i-1+l}(S/I_{G_i}),\]
 thus, $\beta_{n-c,n-c+r-r_i+l}(S/I_G)\geq 1$,
 which is a contradiction to our hypothesis on $I_G$, since $r-r_i+l<r$. Therefore, $I_{G_i}$ is Gorenstein for $1\leq i\leq c$.
\end{proof}

{}

\end{document}